\documentclass[12pt]{article}
\usepackage{amsmath,amssymb,amsthm,graphicx}

\newtheorem{thm}{Theorem}[section] 
 \newtheorem{prop}[thm]{Proposition}
\newtheorem{conj}[thm]{Conjecture}

\begin{document}

\date{\small Mathematics Subject Classification: 91A46}

\title{Rook Endgame Problems in $m$ by $n$ Chess}

\author{Thotsaporn ``Aek'' Thanatipanonda\\
\small{\tt thotsaporn@gmail.com}
}

\maketitle
\begin{abstract}

We consider Chess played on an $m \times n$ board
(with $m$ and $n$ arbitrary positive integers),
with only the two kings and the white rook remaining,
but placed at arbitrary positions.
Using the symbolic finite state method, developed by
Thanatipanonda and Zeilberger, we  prove that
on a $3 \times n$ board, for almost all initial positions,
White can checkmate Black in $\leq n+2$ moves, and that
this upper bound is sharp.
We also conjecture that for an arbitrary $m \times n$ board,
with $m,n \geq 4$ (except for $(m,n)=(4,4)$ when it equals $7$), the number of needed moves is $\leq m+n$,
and that this bound is also sharp.

\end{abstract}

\section{Background and Introduction}

Chess is arguably the most popular board game on this planet.
There are numerous combinatorial problems inspired by Chess,
such as the non-attacking queens problem \cite{wiki1}
and the rich theory of rook polynomials \cite{wiki2} .
More related to the actual game of Chess, Noam Elkies analyzed Chess
endgame positions using the theory of combinatorial games \cite{Elkies}.

Our investigation in this article is in an entirely different direction.
We search for an answer to the following  simply stated question:

{\it ``Find the minimal number of moves needed for White to checkmate against
a perfect opponent when playing
on an $m \times n$ board where the only pieces left are
the two kings and the white rook''}  .\\

Of course, White always moves first.

While the question of the {\it minimal} number of moves needed is not
that important in the usual $8 \times 8$ chess, since the
FIDE rules generously allow $50$ moves, so White has to
be an extremely weak player not be able to win.
On the other hand in Thai Chess, he (or she) is only allowed $16$ moves,
so one has to be much more clever, and the present article may be useful
even for the very special $8 \times 8$ case.


\section{Methodology}

In any rook endgame position on an $m \times n$ board,
White can always checkmate within $O(m+n)$ moves.
The following elegant proof was kindly communicated to us by Noam Elkies.

\begin{prop} \label{Elkies}
In a rook endgame on an $m \times n$ board, the winning side
can checkmate within $O(m+n)$ moves.
\end{prop}
\begin{proof} (Noam Elkies)
Assuming both $m$ and $n$ exceed 2.  [If $m=2$ or $n=2$ there are no checkmates at all; if $m=1$ or $n=1$ then either the Black King is already checkmated or no checkmate is possible -- unless White castles on move 1...] It is true that the method taught in most chess manuals takes time proportional to $m\cdot n$ moves.  But this can be reduced to $O(m+n)$ by using the Rook, supported by the King, to restrict the opposing King to
an $a \times b$ rectangle $(a<m, b<n)$: it takes only $O(1)$ moves to tighten the noose in one dimension or the other, decrementing $a$ to $a-1$ or $b$ to $b-1$; in $O(m+n)$ moves, then, the King will be limited to $O(1)$ squares near the corner, and then checkmate follows in another $O(1)$ moves.
\end{proof}

It would be desirable to have explicit, and if possible, sharp,  upper bounds for the needed number of moves.
We use a symbolic finite state method analogous to the one used in \cite{TT2}.

Let us emphasize the importance of the {\it methodology}, that far transcends the actual results.
i.e. that of rigorous experimental mathematics.
We have a fully implemented algorithm, that completely does the proof without any human intervention!

The computer first generates the data from  many $k \times n$ boards, for
{\it numeric} $k$ and $n$. Then it makes {\it symbolic} conjectures.
Finally the computer automatically proves these general (symbolic) conjectures
{\it all by itself}. We believe that this style of proof
will be more commonly used in the future.

\section {Detail Example}
In this section, we give an example of how
the problem could be done on a $3 \times n$ board.
However, it becomes much harder to handle
wider widths.

\subsection{Example on a $3 \times n$ board}

We fix some ``natural" positions on
a $3 \times n$ board given by

\includegraphics[scale=0.5]{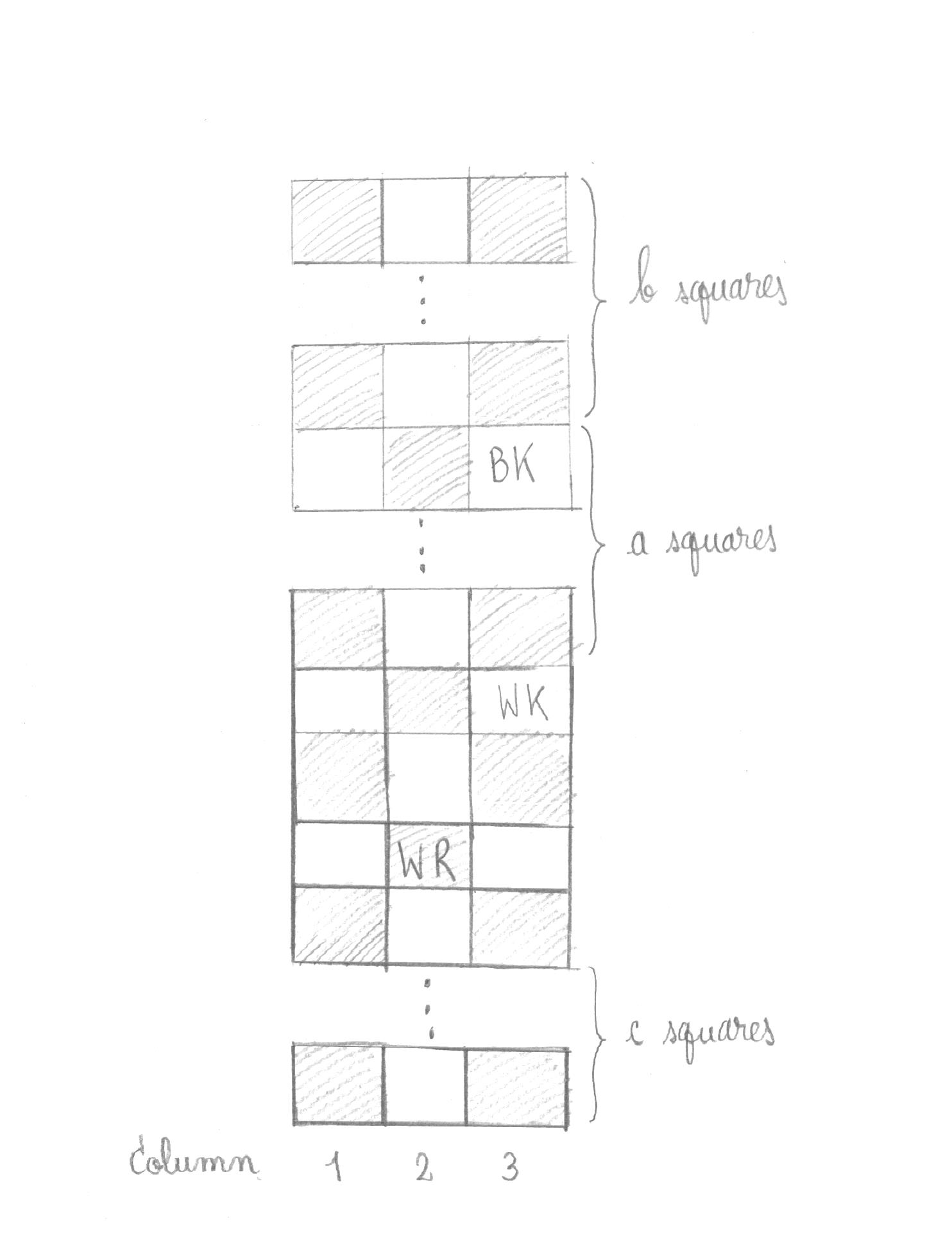}

\noindent The  black King must be above the
white King and the white King must be above the Rook, as shown in the diagram.\\
$a$ is the distance between the black King and the white King.\\
$b$ is the distance of the black King to the end of the board.\\
$c$ is the distance of the Rook to the other end of the board,
but is irrelevant in the calculation.\\

\noindent Let $x$ be the column of the  black King. \\
Let $y$ be the column of the white King. \\
Let $z$ be the column of the white Rook. \\
Let $f_{x,y,z}(a,b)$ be the number of moves needed
to checkmate with the above initial position. \\

We gives sharp upper bounds for all possible
$f_{x,y,z}(a,b), a \geq 0, b \geq 0$.
The proof comprises of three steps, generating data,
making conjectures and proving conjectures. \\

The claims are as follows.

\[  \begin{array}{ll}
 f_{1,1,1}(a,b) &\leq  \;\ a+b+1  \;\ \;\ ,a \geq 2. \\
 f_{1,1,2}(a,b) &\leq \left\{ \begin{array}{ll}
            a+b+1  &,a \;\ \mbox{is even and }a \geq 2. \\
            a+b    &,a \;\ \mbox{is odd and }a \geq 2.
\end{array} \right. \\
f_{1,1,3}(a,b)  &\leq  \;\ a+b+1  \;\ \;\ ,a \geq 2. \\
f_{1,2,2} (a,b)  &\leq  \;\ a+b  \;\ \;\ \;\ \;\ \;\ ,a \geq 2. \\
f_{1,2,3}(a,b)  &\leq \left\{ \begin{array}{ll}
            a+b+2  &,a \;\ \mbox{is even and }a \geq 2. \\
            a+b+1    &,a \;\ \mbox{is odd and }a \geq 2.
\end{array} \right. \\
f_{1,3,2}(a,b) &\leq \left\{ \begin{array}{ll}
            1 &,a = 0. \\
            a+b  &,a \;\ \mbox{is even and }a \geq 1. \\
            a+b+1    &,a \;\ \mbox{is odd and }a \geq 1.\\
\end{array} \right. \\
f_{1,3,3} (a,b) &\leq \left\{ \begin{array}{ll}
            1 &,a = 0. \\
            a+b+2  &,a \;\ \mbox{is even and }a \geq 1. \\
            a+b+1    &,a \;\ \mbox{is odd and }a \geq 1.\\
\end{array} \right. \\
f_{2,1,1} (a,b)  &\leq  \;\ a+b+2  \;\ \;\ ,a \geq 2. \\
f_{2,1,3} (a,b)  &\leq \left\{ \begin{array}{ll}
            a+b+2  &,a \;\ \mbox{is even and }a \geq 2. \\
            a+b+1    &,a \;\ \mbox{is odd and }a \geq 2.
\end{array} \right. \\
f_{2,2,1} (a,b)  &\leq  \;\ a+b+1  \;\ \;\ ,a \geq 2. \\
f_{2,2,2} (a,b)  &\leq \left\{ \begin{array}{ll}
            a+b+1  &,a \;\ \mbox{is even and }a \geq 2. \\
            a+b   &,a \;\ \mbox{is odd and }a \geq 2.
\end{array} \right. \\
\end{array}
\]

Once we got the right set-up, the proof by
induction on $a+b$ falls through naturally.

In order to illustrate our method,
we will now present all the details, in a humanly-readable prose,
for the inductive step for the first assertion above i.e. $f_{1,1,1}(a,b) \leq a+b+1$ ($a \geq 2$).
Of course, this proof was originally discovered by the computer, and the computer quickly
does {\it all} the cases. Once discovered, the computer proof is almost instantaneous, but
{\it discovering} the proposed proof takes longer.

\begin{proof} \textbf{Base Case:} Verify that all the conjectures are true
for $a+b \leq 3 \;\ \mbox{where} \;\ a \geq 0, b \geq 0.$ \\

\textbf{Induction Step:} Consider $f_{1,1,1}(a,b);$ \\

\textbf{Case 1:} $a$ is even.

White chooses to move his Rook to the second column.
Black's only legal moves are to move his King up or down the
first column. By the inductive  hypothesis,
in case the black King moves up, it would take
at most $f_{1,1,2}(a+1,b-1)
= (a+1)+(b-1) = a+b$ moves to checkmate.
In case the black King moves down, it would take
at most $f_{1,1,2}(a-1,b+1)
= (a-1)+(b+1) = a+b$ moves to checkmate.
Therefore case 1 takes at
most $(a+b)+1$ moves to checkmate. \\

\textbf{Case 2:} $a$ is odd.

White chooses to move his King in an up-right direction.
Then the black King will be in check and must move
to the second column. By the induction assumption,
it would take at most max$\{f_{2,2,1}(a-2,b+1),
f_{2,2,1}(a-1,b),f_{2,2,1}(a,b-1)\}=(a+b-1)+1$
 moves to checkmate. Therefore this case takes
 $(a-1+b)+1+1=a+b+1$ moves to checkmate. \\

The proof of case $f_{1,1,1}(a,b)$ is done.
The rest of the proof is left to the readers, if they wish,
but they may prefer to look at the (humanly readable!) computer's full proof at

{\tt http://wannik.com/thotsaporn/Rook.html} .

\end{proof}

\subsection { The Maple programs}
The present method was implemented using Maple.
The two programs, as well as sample output files, are freely available from:

{\tt http://wannik.com/thotsaporn/Rook.html}  \\

\section {Conjecture} \label{conj}

We end this article with an ultimate
conjecture suggested strongly by numerical
evidences for  small $m$ and $n$, given in Table 1. \\

Let $U(m,n)$ be the maximum of the minimal number
of moves needed to checkmate from any initial position of a
rook endgame on an $m \times n$ board.

\begin{table}[ht]
\caption{Values of $U(m,n)$ }
\centering 
\begin{tabular}{|c|c|c|c|c|c|c|c|c|c|c|c|}
  \hline
  $m \backslash n$  & 3 & 4 & 5 & 6 & 7 & 8 & 9 & 10 & 11 & 12 & 13  \\
    \hline
  3   & 3 & 5 & 7  & 8  & 9  & 10 & 11 & 12 & 13 & 14 & 15  \\
  4   & - & 7 & 9  & 10 & 11 & 12 & 13 & 14 & 15 & 16 & 17  \\
  5   & - & - & 10 & 11 & 12 & 13 & 14 & 15 & 16 & 17 & 18  \\
  6   & - & - & -  & 12 & 13 & 14 &  &  & &  &   \\
  7   & - & - & -  & -  & 14 & 15 &  &  & &  &   \\
  8   & - & - & -  & -  &  - & 16 &  &  &  &  &   \\
  \hline
\end{tabular}
\end{table}

\begin{conj} Let $m,n$ be an integer such that
$m \geq 4$ and $n \geq 4$ except $(m,n)=(4,4)$.
In a rook endgame on an $m \times n$ board,
the winning side can checkmate within $m+n$ moves, and
this upper bound is sharp.
\end{conj}

\section {Acknowledgment}
    I thank Doron Zeilberger for introducing me to this
    beautiful, but hard, problem.
    I also thank my twin brother, Thotsaphon Thanatipanonda,
for his help with programming at the beginning of this project.

\end{document}